\documentclass[11pt,reqno]{amsart}

\usepackage[margin=1.1in]{geometry}
\usepackage[T1]{fontenc}
\usepackage{amsmath,amssymb,amsthm,mathtools,mathrsfs, tikz}
\usepackage{microtype}
\usepackage[hidelinks]{hyperref}
\usepackage{orcidlink}
\usepackage[backend=biber,style=alphabetic,maxnames=10,giveninits=true]{biblatex}
\definecolor{vegasgold}{rgb}{0.77, 0.7, 0.35}
\definecolor{darkgoldenrod}{rgb}{0.72, 0.53, 0.04}
\definecolor{gold(metallic)}{rgb}{0.83, 0.69, 0.22}

\hypersetup{
 colorlinks=true,
 linkcolor=darkgoldenrod,
 filecolor=brown,
 urlcolor=gold(metallic),
 citecolor=darkgoldenrod,
}

\usepackage{enumitem}

\numberwithin{equation}{section}

\newtheorem{theorem}{Theorem}[section]
\newtheorem{proposition}[theorem]{Proposition}
\newtheorem{lemma}[theorem]{Lemma}
\newtheorem{corollary}[theorem]{Corollary}

\theoremstyle{definition}

\newcommand{\Q}{\mathbb Q}
\newcommand{\Z}{\mathbb Z}

\newcommand{\Fp}{\mathbb{F}_p}

\newcommand{\Zp}{\mathbb Z_p}

\newcommand{\LambdaIw}{\Lambda}
\newcommand{\Gal}{\operatorname{Gal}}
\newcommand{\Cl}{\operatorname{Cl}}
\newcommand{\rk}{\operatorname{rank}}

\newcommand{\Iw}{\mathrm{Iw}}

\newcommand{\ab}{\mathrm{ab}}
\newcommand{\cF}{\mathcal F}

\title[Iwasawa theory of class field towers]{Iwasawa theory of class field towers}
\author[A.~Ray]{Anwesh Ray\, \orcidlink{0000-0001-6946-1559}}
\address{Chennai Mathematical Institute, H1, SIPCOT IT Park, Siruseri, Kelambakkam 603103, Tamil Nadu, India.}
\email{anwesh@cmi.ac.in}


\begin{document}

\begin{abstract}
Let $p$ be an odd prime and let $K_\infty/K$ be a $\Zp$-extension.
For each finite layer $K_n$, let $G_n$ be the Galois group of its
maximal unramified pro-$p$ extension.  We study the augmentation and
Zassenhaus filtrations of the groups $G_n$ as $n$ tends to infinity.

Assume that the classical Iwasawa $\mu$-invariant is positive.  If
$d_n$ denotes the minimal number of generators of $G_n$, then there is
an integer $\delta>0$ such that $d_n=\delta p^n+O(1)$, while the
minimal number of relations is $O(p^n)$.  For each fixed $m\geq2$, as
$n\to\infty$, we show that the $m$-th augmentation quotient has the
same leading asymptotic as that of a free pro-$p$ group on $d_n$
generators.  We also obtain corresponding asymptotics for the Hilbert
series, the exponential growth rate of the augmentation filtration,
and the Zassenhaus quotients.
\end{abstract}

\maketitle

\section{Introduction}

\subsection{Background and motivation}

Let $K$ be a number field and let $p$ be an odd prime.  Denote by
$H_p(K)$ the maximal abelian unramified $p$-extension of $K$, and put
\[
 A_p(K)=\Gal(H_p(K)/K).
\]
By class field theory, $A_p(K)$ is naturally identified with the
$p$-primary part of the ideal class group of $K$.  If
$h_p(K)=|A_p(K)|$, then $h_p(K)$ is the $p$-part of the class number
of $K$.

Suppose now that $K_\infty/K$ is a $\Zp$-extension, and put
$\Gamma=\Gal(K_\infty/K)$ and $\Lambda=\Zp[[\Gamma]]$.  Let $K_n$ be
the unique intermediate field such that $[K_n:K]=p^n$.  Thus we have
a tower
\[
 K=K_0\subset K_1\subset\cdots\subset K_n\subset K_{n+1}
 \subset\cdots.
\]
If $p^{e_n}=h_p(K_n)$, then Iwasawa's classical theorem asserts that
there are integers $\mu,\lambda\geq0$ and $\nu_{\Iw}\in\Z$ such that
\begin{equation}\label{eq:intro-iwasawa}
 e_n=\mu p^n+\lambda n+\nu_{\Iw}
\end{equation}
for all sufficiently large $n$; see
\cite{Iwasawa1959,Washington1997}.  Thus the $p$-primary class groups
satisfy a precise asymptotic growth law as one moves up the
$\Zp$-extension.

One of the central themes of Iwasawa theory is that the finite-layer
growth in \eqref{eq:intro-iwasawa} is controlled by a single object at
the infinite layer.  If $A_n$ denotes the $p$-primary part of
$\Cl(K_n)$, then
\[
 X=\varprojlim_n A_n
\]
is a finitely generated torsion module over the Iwasawa algebra $\Lambda$.
The invariants $\mu$ and $\lambda$ can be interpreted in terms of the structure of $X$ as a $\Lambda$-module. 

 In this article, we extend this perspective to $p$-class field towers, which are noncommutative in general. Let
$\cF_n$ denote the maximal unramified pro-$p$ extension of $K_n$ and
set
\begin{equation}\label{eq:intro-Gn}
 G_n=\Gal(\cF_n/K_n).
\end{equation}
Equivalently, starting with $H_p^{(0)}(K_n)=K_n$ and
$H_p^{(1)}(K_n)=H_p(K_n)$, one may define inductively
\[
 H_p^{(j)}(K_n)
 =
 H_p\bigl(H_p^{(j-1)}(K_n)\bigr)
\]
for $j\geq2$. The union of these fields is $\cF_n$. Thus $G_n$ is the Galois group
of the full $p$-class field tower of $K_n$.  Class field theory gives
$G_n^{\ab}\simeq A_n$, so Iwasawa's theorem determines the growth of
the abelianization of $G_n$.  It is natural to ask whether there are
corresponding growth laws for the non-abelian groups $G_n$ themselves.

\begin{figure}[t]
\centering
\begin{tikzpicture}[
  scale=0.85,
  transform shape,
  every node/.style={inner sep=1.5pt}
]


\node (K0)   at (0,0)     {$K$};
\node (K1)   at (0,1.5)   {$K_1$};
\node (Kd)   at (0,3.0)   {$\vdots$};
\node (Kn)   at (0,4.5)   {$K_n$};
\node (Knp1) at (0,6.0)   {$K_{n+1}$};
\node (Kup)  at (0,7.5)   {$\vdots$};
\node (Kinf) at (0,9.0)   {$K_\infty$};

\draw (K0)--(K1)--(Kd)--(Kn)--(Knp1)--(Kup)--(Kinf);


\node (H00)   at (2.2,1.27)  {$H_p(K)$};
\node (H01)   at (2.2,2.77)  {$H_p(K_1)$};
\node (H0d)   at (2.2,4.27)  {$\vdots$};
\node (H0n)   at (2.2,5.77)  {$H_p(K_n)$};
\node (H0np1) at (2.2,7.27)  {$H_p(K_{n+1})$};
\node (H0up)  at (2.2,8.77)  {$\vdots$};

\draw (K0)--(H00);
\draw (K1)--(H01);
\draw (Kd)--(H0d);
\draw (Kn)--(H0n);
\draw (Knp1)--(H0np1);
\draw (Kup)--(H0up);

\draw (H00)--(H01)--(H0d)--(H0n)--(H0np1)--(H0up);


\node (H20)   at (4.4,2.54)  {$H_p^{(2)}(K)$};
\node (H21)   at (4.4,4.04)  {$H_p^{(2)}(K_1)$};
\node (H2d)   at (4.4,5.54)  {$\vdots$};
\node (H2n)   at (4.4,7.04)  {$H_p^{(2)}(K_n)$};
\node (H2np1) at (4.4,8.54)  {$H_p^{(2)}(K_{n+1})$};
\node (H2up)  at (4.4,10.04) {$\vdots$};

\draw (H00)--(H20);
\draw (H01)--(H21);
\draw (H0d)--(H2d);
\draw (H0n)--(H2n);
\draw (H0np1)--(H2np1);
\draw (H0up)--(H2up);

\draw (H20)--(H21)--(H2d)--(H2n)--(H2np1)--(H2up);


\node (D0)   at (6.6,3.81)  {$\reflectbox{$\ddots$}$};
\node (D1)   at (6.6,5.31)  {$\reflectbox{$\ddots$}$};
\node (Dd)   at (6.6,6.81)  {$\reflectbox{$\ddots$}$};
\node (Dn)   at (6.6,8.31)  {$\reflectbox{$\ddots$}$};
\node (Dnp1) at (6.6,9.81)  {$\reflectbox{$\ddots$}$};
\node (Dup)  at (6.6,11.31) {$\reflectbox{$\ddots$}$};

\draw (H20)--(D0);
\draw (H21)--(D1);
\draw (H2d)--(Dd);
\draw (H2n)--(Dn);
\draw (H2np1)--(Dnp1);
\draw (H2up)--(Dup);


\node (J0)   at (8.8,5.08)  {$H_p^{(j)}(K)$};
\node (J1)   at (8.8,6.58)  {$H_p^{(j)}(K_1)$};
\node (Jd)   at (8.8,8.08)  {$\vdots$};
\node (Jn)   at (8.8,9.58)  {$H_p^{(j)}(K_n)$};
\node (Jnp1) at (8.8,11.08) {$H_p^{(j)}(K_{n+1})$};
\node (Jup)  at (8.8,12.58) {$\vdots$};

\draw (D0)--(J0);
\draw (D1)--(J1);
\draw (Dd)--(Jd);
\draw (Dn)--(Jn);
\draw (Dnp1)--(Jnp1);
\draw (Dup)--(Jup);

\draw (J0)--(J1)--(Jd)--(Jn)--(Jnp1)--(Jup);


\node (E0)   at (11.0,6.35)  {$\reflectbox{$\ddots$}$};
\node (E1)   at (11.0,7.85)  {$\reflectbox{$\ddots$}$};
\node (Ed)   at (11.0,9.35)  {$\reflectbox{$\ddots$}$};
\node (En)   at (11.0,10.85) {$\reflectbox{$\ddots$}$};
\node (Enp1) at (11.0,12.35) {$\reflectbox{$\ddots$}$};
\node (Eup)  at (11.0,13.85) {$\reflectbox{$\ddots$}$};

\draw (J0)--(E0);
\draw (J1)--(E1);
\draw (Jd)--(Ed);
\draw (Jn)--(En);
\draw (Jnp1)--(Enp1);
\draw (Jup)--(Eup);


\node (F0)   at (13.2,7.62)  {$\cF_0$};
\node (F1)   at (13.2,9.12)  {$\cF_1$};
\node (Fd)   at (13.2,10.62) {$\vdots$};
\node (Fn)   at (13.2,12.12) {$\cF_n$};
\node (Fnp1) at (13.2,13.62) {$\cF_{n+1}$};
\node (Fup)  at (13.2,15.12) {$\vdots$};

\draw (E0)--(F0);
\draw (E1)--(F1);
\draw (Ed)--(Fd);
\draw (En)--(Fn);
\draw (Enp1)--(Fnp1);
\draw (Eup)--(Fup);

\draw (F0)--(F1)--(Fd)--(Fn)--(Fnp1)--(Fup);

\end{tikzpicture}
\caption{The $\Zp$-extension $K_\infty/K$ together with the
$p$-class field towers of its finite layers.}
\label{fig:class-field-towers}
\end{figure}
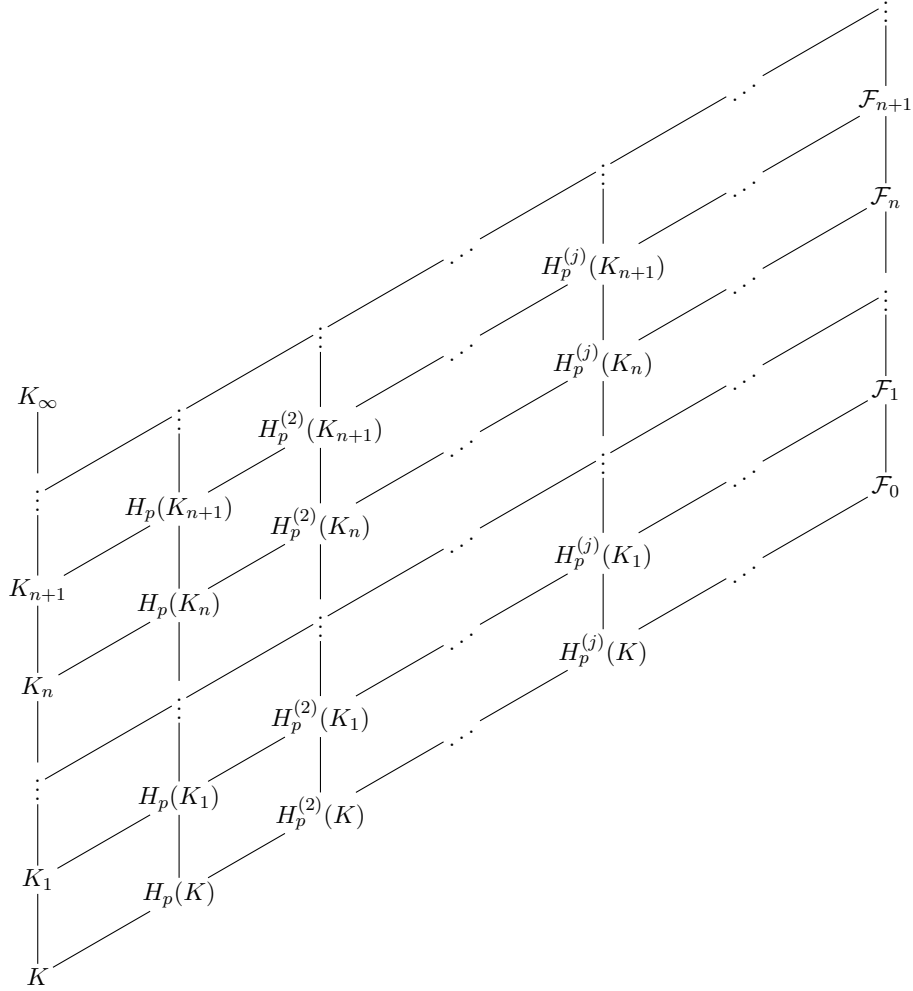

This question was considered from a related point of view in
\cite{BhattacharyyaKadiriRay2024}.  There we studied finitely
ramified pro-$p$ Galois groups as the base field varies through a
$\Zp$-extension and obtained asymptotic lower bounds for their
exponential growth.  The groups considered there may be viewed as
non-abelian analogues of ray class groups.  The present paper returns
to the unramified groups $G_n$ and studies their growth more
precisely.

In studying asymptotic growth properties, we follow the philosophy of
Golod--Shafarevich theory. We consider natural filtrations of $G_n$ and
various generating series associated to them.  This point of view has a
long history in the study of class field towers.  The Golod--Shafarevich inequality relates the number and depth of the defining relations of a pro-$p$ group to the sizes of the successive quotients of the augmentation filtration in its completed group algebra; see \cite{GolodShafarevich1964,Ershov2012}. The key input is Shafarevich's bound on the number of relations needed to present the groups $G_n$ in terms of their number of generators; see Proposition~\ref{prop:shafarevich} and \eqref{eq:rn-explicit}. More refined
questions concerning relation depth and the Zassenhaus filtration
have been studied, for instance, in
\cite{HajirMaire2018,HajirMaire2019,HajirMaireRamakrishna2025}.
\subsection{Main results}
We now describe the invariants considered in this paper.  Let $I_n$
be the augmentation ideal of the completed group algebra
$\Fp[[G_n]]$, and put
\[
 c_{m,n}
 =
 \dim_{\Fp}I_n^m/I_n^{m+1}.
\]
The associated Hilbert series is
\begin{equation}\label{eq:intro-Hn}
 H_n(t)=\sum_{m\geq0}c_{m,n}t^m.
\end{equation}
If $F_d$ is a free pro-$p$ group on $d$ generators,
then the Magnus isomorphism identifies its completed group algebra
with a noncommutative formal power-series algebra on $d$ variables and consequently \[H_{F_d}(t)=\frac{1}{1-dt}.
\]

The Hilbert series also gives a single invariant which measures the
overall exponential growth of the augmentation quotients.  Following
\cite{BhattacharyyaKadiriRay2024}, define
\begin{equation}\label{eq:intro-rho-def}
 \rho(G)
 =
 \limsup_{m\to\infty}
 \left(
 \dim_{\Fp}I_G^m/I_G^{m+1}
 \right)^{1/m}.
\end{equation}
The radius of convergence of the corresponding Hilbert series is
$\rho(G)^{-1}$. The invariant $\rho(G_n)$ measures the
exponential growth of all augmentation degrees simultaneously.

There is a second filtration which is naturally tied to the same
generating series.  Let $D_m(G_n)$ denote the Zassenhaus filtration
and put
\[
 a_{m,n}
 =
 \dim_{\Fp}D_m(G_n)/D_{m+1}(G_n).
\]
The successive quotients form a graded restricted Lie algebra over
$\Fp$. The
augmentation and Zassenhaus filtrations give two different ways of
giving two associated graded invariants of the groups $G_n$.

The results in this paper are of notable interest when the
$\mu$-invariant is positive. Such examples are known for
non-cyclotomic $\Zp$-extensions.  In particular, Iwasawa constructed
families of non-cyclotomic $\Zp$-extensions with positive
$\mu$-invariant; see \cite{IwasawaMu1973,Iwasawa1973}. Leveraging this construction, Hubbard and Washington gave explicit calculations of positive
$\mu$-invariants for several non-cyclotomic $\Zp$-extensions; see \cite{HubbardWashington2018}.  By contrast,
for the cyclotomic $\Zp$-extension of a number field it is conjectured
that the $\mu$-invariant always vanishes.  This was proved by
Ferrero and Washington when the ground field is abelian over $\Q$;
see \cite{FerreroWashington1979}.  Thus the positive-$\mu$ setting
considered here is naturally associated with genuinely
non-cyclotomic phenomena.
  In this case the number of
generators of $G_n$ itself grows exponentially with $n$.  Write $X[p^\infty]$, up to
pseudo-isomorphism, as a direct sum of cyclic $\Lambda$-modules
\begin{equation}\label{eq:intro-elementary}
X[p^\infty]\sim \bigoplus_{j=1}^{\delta}\LambdaIw/(p^{e_j}),
\end{equation}
where $e_j\geq1$.
Then the $\mu$-invariant is given by $\mu=\sum_{j=1}^{\delta}e_j$ and $\delta=\rk_{\Fp[[\Gamma]]}X[p]$.
In particular, $\mu>0$ if and only if $\delta>0$.  

The relevance of the positive $\mu$ assumption to the groups $G_n$
is that it forces their minimal number of generators $d(G_n)$ to grow with the
level. In
Section~\ref{sec:profile}, it is shown that
\begin{equation}\label{eq:intro-generator-growth}
d_n:=d(G_n)=\delta p^n+O(1).
\end{equation}
Thus, when $\mu>0$, one has $\delta>0$ and hence $d_n\asymp p^n$.
At the same time, Shafarevich's estimate bounds the relation rank
$r_n=r(G_n)$ by
\begin{equation}\label{eq:intro-relation-growth}
r_n=O(p^n)=O(d_n).
\end{equation}
Consequently, the number of defining relations grows only linearly in
$d_n$. 
Our first result concerns the individual coefficients of the Hilbert
series.  It shows that, for each fixed $m$, as $n\to\infty$, the $m$-th augmentation
quotient has the same leading term as that of a free pro-$p$ group
with $d_n$ generators.

\begin{theorem}[Theorem~\ref{thm:augmentation}]\label{thm:intro-augmentation}
Assume that $\mu(K_\infty/K)>0$.  Then, for every fixed $m\geq2$,
\begin{equation}\label{eq:intro-augmentation}
 c_{m,n}
 =
 d_n^m+O_m(r_nd_n^{m-2})
 =
 \delta^mp^{mn}+O_m(p^{(m-1)n}).
\end{equation}
For $m=1$, one has $c_{1,n}=d_n=\delta p^n+O(1)$.
\end{theorem}

We also have the following result for the generating series.

\begin{theorem}[Theorem~\ref{thm:hilbert-rate}]\label{thm:intro-hilbert}
Assume that $\mu(K_\infty/K)>0$.  For every $R<\delta^{-1}$,
\begin{equation}\label{eq:intro-hilbert-limit}
 \sup_{|z|\leq R}
 \left|
 H_n(z/p^n)-\frac{1}{1-\delta z}
 \right|
 \ll_R p^{-n}.
\end{equation}
\end{theorem}

Thus, after rescaling the variable by $p^n$, the Hilbert series
converges to the free Hilbert series $(1-\delta z)^{-1}$. This can be interpreted as a uniform version of the asymptotic freeness in
Theorem~\ref{thm:intro-augmentation}.

We next consider the exponential growth number $\rho(G_n)$. In our earlier work
\cite{BhattacharyyaKadiriRay2024}, the main problem was to obtain
lower bounds for this invariant.  In the present unramified setting,
positive $\mu$ gives a much sharper asymptotic.

\begin{theorem}[Theorem~\ref{thm:spectral-strong} and Corollary~\ref{cor:eventual-towers}]\label{thm:intro-spectral}
Assume that $\mu(K_\infty/K)>0$.  Then
\begin{equation}\label{eq:intro-spectral}
 \rho(G_n)
 =
 d_n+O(1)
 =
 \delta p^n+O(1).
\end{equation}
More precisely, for all sufficiently large $n$,
\[
 0\leq d_n-\rho(G_n)\leq\frac{2r_n}{d_n}.
\]
In particular, $G_n$ is infinite for every sufficiently large $n$.
\end{theorem}

The proof again uses only the relative sizes of the generator and
relation ranks.  If a finitely presented pro-$p$ group has $d$
generators and $r$ relations with $d^2>4r$, the quadratic
Golod--Shafarevich estimate gives
\begin{equation}\label{eq:intro-quadratic-rho}
 \frac{d+\sqrt{d^2-4r}}{2}
 \leq
 \rho(G)
 \leq d.
\end{equation}
Since $r_n=O(d_n)$ and $d_n\to\infty$, the difference between the
two bounds remains bounded.  As a consequence, the radius of
convergence of $H_n(t)$ is
\[
 \frac{1}{\delta p^n}+O(p^{-2n}).
\]

Finally, we relate the
augmentation estimates to obtain asymptotics for the Zassenhaus filtration.

\begin{theorem}[Theorem~\ref{thm:zassenhaus-free} and Proposition~\ref{prop:zass2}.]\label{thm:intro-zassenhaus}
Assume that $\mu(K_\infty/K)>0$.  For every fixed $m\geq2$,
\begin{equation}\label{eq:intro-zassenhaus-leading}
 a_{m,n}
 =
 \frac{\delta^m}{m}p^{mn}
 +O_m(p^{(m-1)n}).
\end{equation}
For $m=1$, one has $a_{1,n}=d_n$.  In degree two,
\begin{equation}\label{eq:intro-zassenhaus-two}
 \binom{d_n}{2}-r_n
 \leq
 a_{2,n}
 \leq
 \binom{d_n}{2}.
\end{equation}
\end{theorem}

Although the groups $G_n$ need not themselves be free, the number of defining relations is too small, relative to the number of generators, to affect the leading term in any fixed augmentation or Zassenhaus degree. The results in this article give a non-abelian counterpart to the classical growth patterns of $p$-primary class groups in Iwasawa theory. We expect that the lower-order terms should contain finer arithmetic information. Quadratic relations are the first ones
which can affect the $p^{(m-1)n}$ term in degree $m$, while relations
of larger depth enter only at smaller scales. This is the broad philosophy in
\cite{HajirMaire2018,HajirMaire2019,HajirMaireRamakrishna2025}.

\section{Preliminaries}\label{sec:preliminaries}
In this section, we recall some preliminary notions and set up notation. Throughout, $p$ is an odd prime and $K_\infty/K$ is a $\Zp$-extension.
We write $\Gamma=\Gal(K_\infty/K)$ and denote by $K_n$ the unique
intermediate field such that $[K_n:K]=p^n$.  We choose a topological
generator $\gamma$ of $\Gamma$ and identify the Iwasawa algebra
$\LambdaIw=\Zp[[\Gamma]]$ with $\Zp[[T]]$ by sending $\gamma-1$ to
$T$.  After replacing $K$ by a finite layer, if necessary, we assume
that every prime of $K$ which ramifies in $K_\infty/K$ is totally
ramified.  This replacement does not affect any of the asymptotic
statements considered below.

For a number field $F$, let $\Cl(F)$ denote its ideal class group, and
let $A_n$ be the $p$-primary part of $\Cl(K_n)$.  We denote by
$\cF_n$ the maximal unramified pro-$p$ extension of $K_n$ and put
$G_n=\Gal(\cF_n/K_n)$.  Class field theory identifies the
abelianization of $G_n$ with $A_n$.  We set
\begin{equation}\label{eq:d-r-def}
 d_n=d(G_n):=\dim_{\Fp}H^1(G_n,\Fp)
 \quad\text{and}\quad
 r_n=r(G_n):=\dim_{\Fp}H^2(G_n,\Fp).
\end{equation}
If
\[
 1\longrightarrow R_n\longrightarrow F_n\longrightarrow G_n
 \longrightarrow 1
\]
is a minimal pro-$p$ presentation, then $d_n$ and $r_n$ are,
respectively, the minimal numbers of generators and relations of
$G_n$; see \cite[Propositions~3.9.1 and~3.9.5]{NSW2008}.  Since
$G_n^{\ab}\simeq A_n$, it follows that
\begin{equation}\label{eq:dn-prank}
 d_n=\dim_{\Fp}A_n/pA_n.
\end{equation}

We shall repeatedly use the following relation-rank estimate of
Shafarevich; see \cite[Theorems~11.5 and~11.8]{Koch2002} and
\cite[Theorem~8.7.11 and Corollary~3.9.5]{NSW2008}.

\begin{proposition}[Shafarevich]\label{prop:shafarevich}
Let $F$ be a number field with signature $(r_1(F),r_2(F))$, and let
$G_F^{\mathrm{nr},p}$ be the Galois group of its maximal unramified
pro-$p$ extension.  Then
\begin{equation}\label{eq:shafarevich}
 d(G_F^{\mathrm{nr},p})
 \leq r(G_F^{\mathrm{nr},p})
 \leq d(G_F^{\mathrm{nr},p})
      +r_1(F)+r_2(F)-1+\delta_p(F),
\end{equation}
where $\delta_p(F)=1$ if $F$ contains a primitive $p$-th root of
unity and $\delta_p(F)=0$ otherwise.
\end{proposition}

Since $K_n/K$ is Galois of odd degree $p^n$, no real place of $K$
can become complex in $K_n$.  Thus
$r_i(K_n)=p^nr_i(K)$ for $i=1,2$.  Moreover, the quantity
$\delta_p(K_n)$ is independent of $n$, since
$[\Q(\mu_p):\Q]=p-1$ is prime to $p$.  Applying
Proposition~\ref{prop:shafarevich} to $K_n$ therefore gives
\begin{equation}\label{eq:rn-explicit}
 r_n
 \leq d_n+(r_1(K)+r_2(K))p^n+O(1).
\end{equation}

We next recall the two filtrations of $G_n$ which will be used
throughout the paper.  Let $G$ be a finitely generated pro-$p$ group.
Its completed group algebra over $\Fp$ is the inverse limit
$\Fp[[G]]=\varprojlim_U\Fp[G/U]$, where $U$ ranges over the open
normal subgroups of $G$.  The augmentation map
$\varepsilon_G:\Fp[[G]]\rightarrow\Fp$ is the continuous
$\Fp$-algebra homomorphism determined by $\varepsilon_G(g)=1$ for
$g\in G$, and we denote its kernel by $I_G$.  Since $G$ is finitely
generated, the quotients $I_G^m/I_G^{m+1}$ are finite-dimensional
$\Fp$-vector spaces.  We write
\[
 c_m(G)=\dim_{\Fp}I_G^m/I_G^{m+1}
\]
and denote the corresponding Hilbert series by
$H_G(t)=\sum_{m\geq0}c_m(G)t^m$. We define
\begin{equation}\label{rho(G) defn}\rho(G):=\limsup_{m\to\infty}c_m(G)^{1/m},\end{equation} note that the radius of convergence of $H_G(t)$ is $\rho(G)^{-1}$.

The lower central series of $G$ is defined by $C_1(G)=G$ and
$C_{i+1}(G)=[G,C_i(G)]$, where the commutator subgroup is understood
topologically.  The Zassenhaus filtration combines this commutator
filtration with the $p$-power structure.  If $H$ is a closed subgroup
of $G$, let $H^{p^j}$ denote the closed subgroup generated by the
$p^j$-th powers of elements of $H$.  Then
\begin{equation}\label{eq:zass-def}
 D_m(G)
 =
 \prod_{ip^j\geq m} C_i(G)^{p^j},
\end{equation}
where the product denotes the closed subgroup generated by the
indicated factors.  We put
\begin{equation}\label{eq:am-def}
 a_m(G)=\dim_{\Fp}D_m(G)/D_{m+1}(G).
\end{equation}
The Jennings--Lazard theorem relates the augmentation filtration to
the Zassenhaus filtration and gives the identity
\begin{equation}\label{eq:jennings}
 H_G(t)
 =
 \prod_{j\geq1}
 \left(\frac{1-t^{pj}}{1-t^j}\right)^{a_j(G)}.
\end{equation}
See \cite{Jennings1941,Lazard1965} and
\cite[\S7.4]{Koch2002}; see also \cite{MinacRogelstadTan2016} for the
form of the identity used here.

The graded vector space
\[
 L(G)=\bigoplus_{m\geq1}D_m(G)/D_{m+1}(G)
\]
has a natural structure of a graded restricted Lie algebra over
$\Fp$: the bracket is induced by commutators in $G$, while the
restricted $p$-operation is induced by the map $x\mapsto x^p$.
Indeed, the inclusions $[D_i(G),D_j(G)]\subseteq D_{i+j}(G)$ and
$D_i(G)^p\subseteq D_{pi}(G)$ ensure that these operations are
well-defined on the associated graded object.

We shall only require this construction in the free case.  Let $F_d$
be the free pro-$p$ group on generators $x_1,\ldots,x_d$.  Lazard
proved that $L(F_d)$ is the free restricted Lie algebra over $\Fp$ on
the initial forms of these generators; see
\cite[Theorem~6.5]{Lazard1954}.  We write
\[
 w_m^{(p)}(d)
 =
 \dim_{\Fp}D_m(F_d)/D_{m+1}(F_d)
\]
for the dimension of its degree-$m$ component.  The following
elementary consequences of the restricted
Poincar\'e--Birkhoff--Witt theorem will be used below.

\begin{lemma}\label{lem:free-restricted-lie-dimensions}
With the notation above, the following assertions hold.
\begin{enumerate}
\item The dimensions $w_j^{(p)}(d)$ satisfy
\begin{equation}\label{eq:free-restricted-hilbert}
 \frac{1}{1-dt}
 =
 \prod_{j\geq1}
 \left(\frac{1-t^{pj}}{1-t^j}\right)^{w_j^{(p)}(d)}.
\end{equation}

\item For every fixed $m\geq1$,
\begin{equation}\label{eq:free-restricted-asymptotic}
 w_m^{(p)}(d)
 =
 \frac{d^m}{m}+O_m(d^{m-1})
\end{equation}
as $d\to\infty$.

\item If $m<p$, then
\begin{equation}\label{eq:free-restricted-witt}
 w_m^{(p)}(d)
 =
 \frac{1}{m}
 \sum_{e\mid m}\mu_{\mathrm{Mob}}(e)d^{m/e},
\end{equation}
where $\mu_{\mathrm{Mob}}$ denotes the M\"obius function.

\item Since $p$ is odd, one has
$w_2^{(p)}(d)=\binom{d}{2}$.
\end{enumerate}
\end{lemma}

\begin{proof}
By the restricted Poincar\'e--Birkhoff--Witt theorem, after choosing
an ordered homogeneous basis of a restricted Lie algebra, its
restricted universal enveloping algebra has a basis consisting of
ordered monomials in which every basis element occurs with exponent
between $0$ and $p-1$; see
\cite[Chapter~V, \S7]{Jacobson1962} or \cite[\S12.1]{DDMS1999}.
Thus a homogeneous basis element of degree $j$ contributes the factor
$1+t^j+\cdots+t^{(p-1)j}=(1-t^{pj})/(1-t^j)$ to the Hilbert series.

For the free restricted Lie algebra $L(F_d)$, its restricted universal
enveloping algebra is the free associative algebra on the initial
forms of $x_1,\ldots,x_d$.  Its degree-$n$ component has dimension
$d^n$, and hence its Hilbert series is $(1-dt)^{-1}$.  Multiplying the
PBW contributions of the homogeneous basis elements therefore gives
\eqref{eq:free-restricted-hilbert}.

Taking formal logarithms of \eqref{eq:free-restricted-hilbert} and
comparing the coefficient of $t^m$ gives
\begin{equation}\label{eq:free-restricted-recursion}
 d^m
 =
 \sum_{j\mid m}j\,w_j^{(p)}(d)
 -
 p\sum_{pj\mid m}j\,w_j^{(p)}(d).
\end{equation}
We prove the second assertion by induction on $m$.  The assertion is
clear when $m=1$.  For $m>1$, isolate the term corresponding to $j=m$
in the first sum in \eqref{eq:free-restricted-recursion}.  Every other
index which occurs is strictly smaller than $m$, and the induction
hypothesis therefore shows that the remaining terms contribute
$O_m(d^{m-1})$.  It follows that
$m w_m^{(p)}(d)=d^m+O_m(d^{m-1})$, which proves (2).

Suppose now that $m<p$.  There is no positive integer $j$ for which
$pj$ divides $m$, so the second sum in
\eqref{eq:free-restricted-recursion} is empty.  We therefore have
$d^m=\sum_{j\mid m}j\,w_j^{(p)}(d)$.  M\"obius inversion gives
\[
 m\,w_m^{(p)}(d)
 =
 \sum_{e\mid m}\mu_{\mathrm{Mob}}(e)d^{m/e},
\]
which proves (3).  Finally, taking $m=2$ and using that $p$ is odd
gives $w_2^{(p)}(d)=(d^2-d)/2=\binom{d}{2}$.
\end{proof}

We briefly recall the Golod--Shafarevich theorem.  Let
\begin{equation}\label{eq:min-presentation}
 1\longrightarrow R\longrightarrow F\longrightarrow G\longrightarrow1
\end{equation}
be a minimal presentation, where $F$ is free pro-$p$ on $d$ generators and $R$ is normally generated by $r$ relators.  Minimality implies $R\subseteq D_2(F)$.  If $r_k$ relators have Zassenhaus depth $k$, then the Golod--Shafarevich polynomial is defined as follows:
\begin{equation}\label{eq:GS-poly}
 P_G(t)=1-dt+\sum_{k\ge2}r_kt^k.
\end{equation}
\begin{proposition}[Golod--Shafarevich]
\label{prop:GS-spectral-bound}
With respect to notation above, the following
assertions hold.
\begin{enumerate}
\item If there exists $t_0\in(0,1)$ such that $P_G(t_0)<0$, then the
Hilbert series $H_G(t)$ does not converge at $t=t_0$.  Consequently, $\rho(G)
 \ge t_0^{-1}$.

\item One has, for every $0<t<1$,
\begin{equation}\label{eq:quadratic-majorant}
 P_G(t)
 \le
 1-dt+rt^2.
\end{equation}
\end{enumerate}
\end{proposition}

\begin{proof}
The first assertion is the Golod--Shafarevich inequality.  If
$P_G(t_0)<0$ for some $t_0\in(0,1)$, then the inequality forces the
Hilbert series $H_G(t)$ to diverge at $t=t_0$; see
\cite{GolodShafarevich1964,Ershov2012}.  The radius of convergence of $H_G(t)$ is $\rho(G)^{-1}$.
Since $H_G(t)$ does not converge at $t_0$, its radius of convergence
is at most $t_0$, and therefore $\rho(G)\ge t_0^{-1}$.

For the second assertion, minimality of the presentation implies that
no defining relation has Zassenhaus depth one, so $r_k=0$ for $k<2$. Hence
\[
 \sum_{k\ge2}r_kt^k
 \le
 t^2\sum_{k\ge2}r_k
 =
 rt^2,
\]
which gives \eqref{eq:quadratic-majorant}.
\end{proof}

\section{The Iwasawa module and the generator rank}
\label{sec:profile}

Recall that $X=\varprojlim_n A_n$, where the transition maps are the norm maps. This is the classical
unramified Iwasawa module associated with $K_\infty/K$.  The action of
$\Gamma=\Gal(K_\infty/K)$ makes $X$ a finitely generated torsion module
over $\LambdaIw=\Zp[[\Gamma]]$.  With the choice of a topological
generator $\gamma$ made in Section~\ref{sec:preliminaries}, we identify
$\LambdaIw$ with $\Zp[[T]]$ by $\gamma-1\mapsto T$.

We recall the structure theorem for finitely generated torsion modules
over the Iwasawa algebra; see
\cite[Theorem~13.12]{Washington1997}.  Up to pseudo-isomorphism,
\begin{equation}\label{eq:X-elementary-decomposition}
 X
 \sim
 \bigoplus_{j=1}^{\delta}\LambdaIw/(p^{e_j})
 \oplus
 \bigoplus_{\ell=1}^{t}
 \LambdaIw/(f_\ell(T)^{a_\ell}),
\end{equation}
where $e_j\geq1$, $a_\ell\geq1$, and each $f_\ell(T)$ is an
irreducible distinguished polynomial. We have that
\[
 \mu=\sum_{j=1}^{\delta}e_j
 \quad\text{and}\quad
 \delta=\rk_{\Fp[[\Gamma]]}X[p].
\]
In particular, $\mu>0$ if and only if $\delta>0$.

For the arguments below, we only need to determine the growth of the
$p$-rank of $A_n$.  Recall from \eqref{eq:dn-prank} that
\[
 d_n=d(G_n)=\dim_{\Fp}A_n/pA_n.
\]
The following consequence of the structure theorem gives the required
estimate.

\begin{proposition}
\label{prop:finite-layer-profile}
With the notation above, one has
\begin{equation}\label{eq:dn-exact}
 d_n=\delta p^n+O(1).
\end{equation}
In particular, $\mu>0$ if and only if $d_n$ grows linearly with
$[K_n:K]=p^n$.
\end{proposition}

\begin{proof}
This is a standard consequence of the structure theory of Iwasawa
modules. More generally, Perbet proves that for a uniform $p$-adic Lie
extension of dimension $d$, the $p$-rank of the class group at the
$n$-th layer satisfies
\[
 d_pA(K_n)
 =
 r(X)p^{dn}+O\bigl(p^{n(d-1)}\bigr),
\]
where
\[
 r(X)=\operatorname{rank}_{\Fp[[\Gamma]]}X[p].
\]
See \cite{Perbet2011}; see also
\cite[Theorem~1.7 and Corollary~1.8]{HajirMaire2019}.
\end{proof}
\section{Augmentation quotients and asymptotic freeness}
\label{sec:asymptotic-freeness}

The purpose of this section is to compare the low-degree augmentation
quotients of $G_n$ with those of a free pro-$p$ group having the same
number of generators.
\begin{lemma}
\label{lem:fixed-degree}
For each $n$, let $G_n$ have a minimal pro-$p$ presentation with
$d_n$ generators and $r_n$ relations. Setting \[S_m(d_n):=\sum_{q=2}^m(m-q+1)d_n^{m-q},\] we find that for every $m\geq2$,
\begin{equation}\label{eq:fixed-bound}
 d_n^m-r_nS_m(d_n)
 \leq c_{m,n}\leq d_n^m.
\end{equation}
Consequently, if $d_n\to\infty$ and $r_n=O(d_n)$, then, for every
fixed $m$,
\begin{equation}\label{eq:fixed-asymp}
 c_{m,n}=d_n^m+O_m(d_n^{m-1}).
\end{equation}
\end{lemma}

\begin{proof}
Fix $n$, and choose a minimal pro-$p$ presentation
\[
 1\longrightarrow R_n\longrightarrow \mathscr F_n
 \longrightarrow G_n\longrightarrow1,
\]
where $\mathscr F_n$ is the free pro-$p$ group on
$x_1,\ldots,x_{d_n}$ and $R_n$ is normally generated by
$r_n$ elements.  Let $\mathfrak m_n$ denote the augmentation ideal
of $\Fp[[\mathscr F_n]]$.  The Magnus isomorphism identifies
$\Fp[[\mathscr F_n]]$ with the algebra of noncommutative formal
power series
$\Fp\langle\!\langle X_1,\ldots,X_{d_n}\rangle\!\rangle$ by sending
$x_i$ to $1+X_i$.  Under this identification, $\mathfrak m_n$
corresponds to the ideal of series with zero constant term.
Consequently,
$\mathfrak m_n^j/\mathfrak m_n^{j+1}$ is naturally the space of
homogeneous noncommutative polynomials of degree $j$ in the variables
$X_1,\ldots,X_{d_n}$.  Its $\Fp$-dimension is therefore $d_n^j$.

Choose relators
$\rho_{1,n},\ldots,\rho_{r_n,n}$ which minimally normally generate
$R_n$.  Let $J_n$ denote the closed two-sided ideal of
$\Fp[[\mathscr F_n]]$ generated by the elements
$\rho_{i,n}-1$.  Since the presentation is minimal, the image of every
$\rho_{i,n}$ in
$\mathscr F_n/\Phi(\mathscr F_n)$ is trivial.  For a pro-$p$ group,
$\Phi(\mathscr F_n)=D_2(\mathscr F_n)$, and the dimension-subgroup
description of the Zassenhaus filtration gives
$\rho_{i,n}-1\in\mathfrak m_n^2$.  Thus its Magnus expansion has the
form
\[
 \rho_{i,n}-1=f_{i,n,2}+f_{i,n,3}+\cdots,
\]
where each $f_{i,n,q}$ is homogeneous of degree $q$.

We now determine how much the ideal $J_n$ can remove from degree
$m$.  Since we work modulo $\mathfrak m_n^{m+1}$, only finitely many
homogeneous pieces of the Magnus expansions are relevant.  Moreover,
the image of the closure of $J_n$ in the finite-dimensional algebra
$\Fp[[\mathscr F_n]]/\mathfrak m_n^{m+1}$ is the ordinary ideal
generated by the images of the $\rho_{i,n}-1$.

The degree-$m$ part of that ideal is spanned by expressions
$u f_{i,n,q}v$, where $2\leq q\leq m$, and where $u$ and $v$ are
homogeneous noncommutative polynomials whose degrees add up to
$m-q$.  Fix $i$ and $q$, and write $a=\deg u$.  There are
$d_n^a$ monomials of degree $a$ and
$d_n^{m-q-a}$ monomials of degree $m-q-a$.  Hence the subspace
obtained for this fixed value of $a$ is spanned by at most
$d_n^{m-q}$ elements.  Since $a$ can take the $m-q+1$ values
$0,\ldots,m-q$, the contribution of the degree-$q$ part of one
relator has dimension at most $(m-q+1)d_n^{m-q}$. We find that
\[
 \dim_{\Fp}
 \frac{(J_n\cap\mathfrak m_n^m)+\mathfrak m_n^{m+1}}
      {\mathfrak m_n^{m+1}}
 \leq
 r_n\sum_{q=2}^m(m-q+1)d_n^{m-q}.
\]

Since
$\Fp[[G_n]]\simeq\Fp[[\mathscr F_n]]/J_n$, the $m$-th augmentation
quotient of $\Fp[[G_n]]$ is obtained from
$\mathfrak m_n^m/\mathfrak m_n^{m+1}$ by quotienting by the subspace
displayed above.  The free space has dimension $d_n^m$, and therefore
\eqref{eq:fixed-bound} follows.

For fixed $m$, the highest power of $d_n$ occurring in $S_m(d_n)$
comes from $q=2$ and is $d_n^{m-2}$.  Hence
$S_m(d_n)=O_m(d_n^{m-2})$.  If $r_n=O(d_n)$, the total possible loss
in degree $m$ is consequently $O_m(d_n^{m-1})$, which proves
\eqref{eq:fixed-asymp}.
\end{proof}

The preceding estimate may also be summed over all degrees, provided
one stays inside the disc on which the free Hilbert series converges. Recall that $H_n(t)=\sum_{m\geq0}c_{m,n}t^m$ is the Hilbert series
of the augmentation filtration of $\Fp[[G_n]]$.

\begin{lemma}\label{lem:hilbert-defect}
Assume $d_n\geq1$.  For
$0\leq t<d_n^{-1}$ one has
\begin{equation}\label{eq:hilbert-defect}
 0\leq\frac{1}{1-d_nt}-H_n(t)
 \leq
 \frac{r_nt^2}{(1-t)(1-d_nt)^2}.
\end{equation}
\end{lemma}

\begin{proof}
Lemma~\ref{lem:fixed-degree} gives, coefficient by coefficient,
$0\leq d_n^m-c_{m,n}\leq r_nS_m(d_n)$ for $m\geq2$.  In degrees
zero and one there is no loss, as the constant coefficient is one and
$c_{1,n}=d_n$.  Since $d_nt<1$, both the free series and all the
majorizing series below converge absolutely.  We may therefore sum
the inequalities to obtain
\[
 0\leq\frac{1}{1-d_nt}-H_n(t)
 \leq r_n\sum_{m\geq2}S_m(d_n)t^m.
\]

It remains to evaluate the series on the right.  A term contributing
to $S_m(d_n)$ consists of a homogeneous piece of a relator of degree
$q\geq2$, together with a word of degree $a$ on its left and a word
of degree $b$ on its right, where $m=q+a+b$.  Hence
\[
 \sum_{m\geq2}S_m(d_n)t^m
 =
 \sum_{q\geq2}\sum_{a,b\geq0}
 d_n^{a+b}t^{q+a+b}
 =
 \frac{t^2}{(1-t)(1-d_nt)^2}.
\]
Substitution proves \eqref{eq:hilbert-defect}.
\end{proof}
We now record the two estimates from Iwasawa theory and class field
theory that will be used in the augmentation arguments below.

\begin{proposition}\label{prop:dr-growth}
Assume $\mu>0$.  Then, for all sufficiently large $n$,
\begin{equation}\label{eq:dn-delta}
 d_n=\delta p^n+O(1),
\end{equation}
and
\begin{equation}\label{eq:rn-linear}
 r_n
 \leq d_n+(r_1(K)+r_2(K))p^n+O(1)
 =O(p^n)=O(d_n).
\end{equation}
In particular,
\begin{equation}\label{eq:r-over-d2}
 \frac{r_n}{d_n^2}=O(p^{-n}),
\end{equation}
and
\begin{equation}\label{eq:r-over-d-limsup}
 \limsup_{n\to\infty}\frac{r_n}{d_n}
 \leq
 1+\frac{r_1(K)+r_2(K)}{\delta}.
\end{equation}
\end{proposition}

\begin{proof}
By \eqref{eq:dn-prank}, the number $d_n$ is the $p$-rank of $A_n$.
Proposition~\ref{prop:finite-layer-profile} therefore gives
\eqref{eq:dn-delta}.  Since $\mu>0$, one has $\delta>0$, and hence
$d_n\asymp p^n$.

The relation-rank estimate \eqref{eq:rn-explicit} gives
\[
 r_n\leq d_n+(r_1(K)+r_2(K))p^n+O(1).
\]
As both terms on the right are $O(p^n)$, this proves
$r_n=O(p^n)=O(d_n)$.  Dividing by $d_n^2\asymp p^{2n}$ gives
\eqref{eq:r-over-d2}.  Finally, divide the same relation-rank bound by
$d_n=\delta p^n+O(1)$ and let $n$ tend to infinity.  The first term
contributes $1$, while the second contributes at most
$(r_1(K)+r_2(K))/\delta$, proving
\eqref{eq:r-over-d-limsup}.
\end{proof}

\begin{theorem}\label{thm:augmentation}
Assume $\mu>0$.  For every fixed $m\geq2$,
\begin{equation}\label{eq:augmentation-explicit}
 0\leq d_n^m-c_{m,n}\leq r_nS_m(d_n).
\end{equation}
Consequently,
\begin{equation}\label{eq:augmentation}
 c_{m,n}
 =
 d_n^m+O_m(r_nd_n^{m-2})
 =
 d_n^m+O_m(d_n^{m-1})
 =
 \delta^mp^{mn}+O_m(p^{(m-1)n}).
\end{equation}
For $m=1$, one has $c_{1,n}=d_n=\delta p^n+O(1)$.
\end{theorem}

\begin{proof}
The first inequality is precisely Lemma~\ref{lem:fixed-degree}.
For fixed $m$, we have
$S_m(d_n)=O_m(d_n^{m-2})$, and therefore
$d_n^m-c_{m,n}=O_m(r_nd_n^{m-2})$.  Proposition~\ref{prop:dr-growth}
gives $r_n=O(d_n)$, which yields the second error estimate. Since $d_n=\delta p^n+O(1)$, the binomial theorem gives
$d_n^m=\delta^mp^{mn}+O_m(p^{(m-1)n})$ for fixed $m$.  Combining
the two estimates proves \eqref{eq:augmentation}.  In degree one,
$I_n/I_n^2$ is the Frattini quotient of $G_n$ after tensoring with
$\Fp$, so its dimension is $d_n$ by definition.
\end{proof}

The preceding theorem gives the asymptotic behavior in each fixed
degree.  We now combine these estimates to obtain a uniform statement
for the Hilbert series after the natural rescaling by $p^n$.

\begin{theorem}
\label{thm:hilbert-rate}
Assume $\mu>0$.  For every $R<\delta^{-1}$,
\begin{equation}\label{eq:hilbert-uniform}
 \sup_{|z|\leq R}
 \left|
 H_n(z/p^n)-\frac{1}{1-\delta z}
 \right|
 \ll_R p^{-n}.
\end{equation}
\end{theorem}

\begin{proof}
Put $u_n=d_n/p^n$.  Equation \eqref{eq:dn-delta} gives
$u_n=\delta+O(p^{-n})$.  Fix $R<\delta^{-1}$.  Since
$u_n\to\delta$, after discarding finitely many layers there is an
$\eta<1$ such that $u_nR\leq\eta$ for every remaining $n$.

We first compare $H_n(z/p^n)$ with the Hilbert series of the free
pro-$p$ group on $d_n$ generators.  For $|z|\leq R$, absolute
convergence follows from $u_n|z|\leq\eta$.  Using the coefficient
inequality of Lemma~\ref{lem:fixed-degree}, and then summing exactly
as in Lemma~\ref{lem:hilbert-defect}, we obtain
\begin{align*}
 \left|
 \frac{1}{1-u_nz}-H_n(z/p^n)
 \right|
 &\leq
 \sum_{m\geq2}
 (d_n^m-c_{m,n})\frac{|z|^m}{p^{mn}} \\
 &\leq
 \frac{r_n|z|^2p^{-2n}}
 {(1-|z|p^{-n})(1-u_n|z|)^2}.
\end{align*}
For $|z|\leq R$, the first denominator tends uniformly to one, while
the second is at least $(1-\eta)^2$.  Since
$r_n=O(p^n)$ by Proposition~\ref{prop:dr-growth}, the entire
expression is $O_R(p^{-n})$.

It remains to compare the two free series.  Again uniformly for
$|z|\leq R$,
\[
 \left|
 \frac{1}{1-u_nz}-\frac{1}{1-\delta z}
 \right|
 =
 \frac{|u_n-\delta||z|}
 {|1-u_nz|\,|1-\delta z|}
 \ll_R p^{-n}.
\]
The triangle inequality now proves \eqref{eq:hilbert-uniform}.
\end{proof}

We next turn from fixed augmentation degrees to the exponential
growth of all degrees simultaneously.  Recall that
$\rho(G_n)$ is the reciprocal of the radius of convergence of
$H_n(t)$.

\begin{proposition}
\label{prop:quadratic-rho}
Let $G$ be a finitely presented pro-$p$ group with a minimal
presentation on $d$ generators and $r$ relations.  Suppose
$d^2>4r$.  Then
\begin{equation}\label{eq:rho-quadratic}
 \frac{d+\sqrt{d^2-4r}}{2}
 \leq \rho(G)\leq d.
\end{equation}
Consequently,
\begin{equation}\label{eq:rho-defect}
 0\leq d-\rho(G)
 \leq
 \frac{d-\sqrt{d^2-4r}}{2}
 =
 \frac{2r}{d+\sqrt{d^2-4r}}
 \leq\frac{2r}{d}.
\end{equation}
\end{proposition}

\begin{proof}
The upper bound is immediate from the free presentation.  Indeed,
the completed group algebra of $G$ is a quotient of the completed
group algebra of the free pro-$p$ group on $d$ generators, and hence
$c_m(G)\leq d^m$ for every $m$.  Taking $m$-th roots and upper limits
gives $\rho(G)\leq d$.

If $r=0$, the group is free and equality holds, so suppose $r>0$.
The roots of the quadratic polynomial $1-dt+rt^2$ are
\[
 t_\pm=\frac{d\pm\sqrt{d^2-4r}}{2r}
 \quad\text{and}\quad
 t_-=\frac{2}{d+\sqrt{d^2-4r}}.
\]
The assumption $d^2>4r$ makes these roots real and distinct.
Moreover, except for the already treated free case, it forces
$d\geq2$, and one has $0<t_-<1$.  The quadratic is negative on
$(t_-,t_+)$.  Therefore, for every
$t\in(t_-,\min\{t_+,1\})$, the quadratic majorant
\eqref{eq:quadratic-majorant} gives $P_G(t)<0$.
Proposition~\ref{prop:GS-spectral-bound} then implies
$\rho(G)\geq t^{-1}$.  Letting $t$ decrease to $t_-$ gives the lower
bound in \eqref{eq:rho-quadratic}.

Subtracting this lower bound from $d$ and rationalizing gives
\[
 d-\rho(G)
 \leq
 \frac{d-\sqrt{d^2-4r}}{2}
 =
 \frac{2r}{d+\sqrt{d^2-4r}},
\]
and the denominator is at least $d$.  This proves
\eqref{eq:rho-defect}.

We shall also use the corresponding first-order expansion.  If
$r/d^2\to0$, then
\begin{equation}\label{eq:rho-defect-expansion}
 d-\rho(G)
 \leq
 \frac{r}{d}
 +O\left(\frac{r^2}{d^3}\right).
\end{equation}
This is obtained by expanding
$\sqrt{1-4r/d^2}$ at the origin.
\end{proof}

\begin{theorem}
\label{thm:spectral-strong}
Assume $\mu>0$.  Then
\begin{equation}\label{eq:rho-bounded}
 \rho(G_n)=d_n+O(1)=\delta p^n+O(1).
\end{equation}
More explicitly, for all sufficiently large $n$,
\begin{equation}\label{eq:rho-explicit-error}
 0\leq d_n-\rho(G_n)\leq\frac{2r_n}{d_n}.
\end{equation}
Moreover,
\begin{equation}\label{eq:rho-defect-limsup}
 \limsup_{n\to\infty}
 \bigl(d_n-\rho(G_n)\bigr)
 \leq
 1+\frac{r_1(K)+r_2(K)}{\delta}.
\end{equation}
\end{theorem}

\begin{proof}
Proposition~\ref{prop:dr-growth} gives $r_n=O(d_n)$ and
$d_n\to\infty$.  Hence $r_n/d_n^2\to0$, so in particular
$d_n^2>4r_n$ for all sufficiently large $n$.  Applying
Proposition~\ref{prop:quadratic-rho} gives
\[
 0\leq d_n-\rho(G_n)\leq\frac{2r_n}{d_n}=O(1).
\]
This proves the first equality in \eqref{eq:rho-bounded}; the second
follows from $d_n=\delta p^n+O(1)$.

For the sharper limsup, use
\eqref{eq:rho-defect-expansion}.  Since $r_n=O(d_n)$ and
$d_n\asymp p^n$, its error term satisfies
$r_n^2/d_n^3=O(p^{-n})$.  Therefore
\begin{equation}\label{eq:rho-arithmetic-refined}
 0\leq d_n-\rho(G_n)
 \leq\frac{r_n}{d_n}+O(p^{-n}).
\end{equation}
Taking upper limits and applying
\eqref{eq:r-over-d-limsup} proves
\eqref{eq:rho-defect-limsup}.
\end{proof}

\begin{corollary}\label{cor:eventual-towers}
If $\mu(K_\infty/K)>0$, then $G_n$ is infinite for every sufficiently
large $n$.  Moreover, the radius of convergence of $H_n(t)$ is
\[
 \frac{1}{\delta p^n}+O(p^{-2n}).
\]
\end{corollary}

\begin{proof}
The lower bound in \eqref{eq:rho-quadratic} tends to infinity with
$n$, and in particular is greater than one for all sufficiently
large $n$.  A finite $p$-group has nilpotent augmentation ideal over
$\Fp$, so its Hilbert series is a polynomial and its corresponding
growth invariant is zero.  Thus $\rho(G_n)>1$ forces $G_n$ to be
infinite.

Since
\[
 \rho(G_n)=\delta p^n+O(1),
\]
we may write $\rho(G_n)=\delta p^n+\varepsilon_n$ with
$\varepsilon_n=O(1)$.  Therefore
\[
 \rho(G_n)^{-1}
 =
 \frac{1}{\delta p^n}
 \left(1+\frac{\varepsilon_n}{\delta p^n}\right)^{-1}.
\]
As $\varepsilon_n/(\delta p^n)=O(p^{-n})$, the expansion
$(1+x)^{-1}=1+O(x)$ gives
\[
 \rho(G_n)^{-1}
 =
 \frac{1}{\delta p^n}
 +O(p^{-2n}).
\]
\end{proof}

\section{Zassenhaus growth}\label{sec:zassenhaus}

We now turn to the Zassenhaus filtration.  Recall from
\eqref{eq:jennings} in Section~\ref{sec:preliminaries} that the
Jennings--Lazard identity gives
\[
 H_n(t)
 =
 \prod_{j\geq1}
 \left(\frac{1-t^{pj}}{1-t^j}\right)^{a_{j,n}},
\]
where
\[
 a_{j,n}
 =
 \dim_{\Fp}D_j(G_n)/D_{j+1}(G_n).
\]
We use this identity to compare the Zassenhaus quotients of $G_n$
with those of a free pro-$p$ group.  We begin by studying the
logarithm of $H_n(t)$. For a formal power series $F(t)$, we write $[t^m]F(t)$ for the
coefficient of $t^m$ in $F(t)$.
\begin{lemma}\label{lem:log-H}
Assume $\mu>0$.  For every fixed $m\geq2$,
\begin{equation}\label{eq:log-H}
 [t^m]\log H_n(t)
 =
 \frac{d_n^m}{m}
 +O_m(r_nd_n^{m-2}).
\end{equation}
For $m=1$, the coefficient is $d_n$.
\end{lemma}

\begin{proof}
For $j\geq2$, write
\[
 c_{j,n}=d_n^j-\varepsilon_{j,n},
\]
and $\varepsilon_{1,n}=0$. By Theorem~\ref{thm:augmentation}, for each fixed $m$ there is a
constant $C_m$ such that
\begin{equation}\label{eq:epsilon-bound}
 0\leq\varepsilon_{j,n}
 \leq C_m r_nd_n^{j-2}
\end{equation}
for $2\leq j\leq m$. We shall compare the coefficient of $\log H_n(t)$ with the
corresponding coefficient for the free Hilbert series.

Since
\[
 H_n(t)=1+\sum_{j\geq1}c_{j,n}t^j,
\]
the formal identity
$\log(1+X)=\sum_{k\geq1}(-1)^{k+1}X^k/k$ gives
\begin{equation}\label{eq:log-coefficient-compositions}
 [t^m]\log H_n(t)
 =
 \sum_{k=1}^m\frac{(-1)^{k+1}}{k}
 \sum_{\substack{j_1+\cdots+j_k=m\\ j_\nu\geq1}}
 \prod_{\nu=1}^k c_{j_\nu,n}.
\end{equation}
For a free pro-$p$ group on $d_n$ generators one has
$c_j^{\mathrm{free}}=d_n^j$ for every $j$, and hence
\[
 \log\frac{1}{1-d_nt}
 =
 \sum_{m\geq1}\frac{d_n^m}{m}t^m.
\]
Applying \eqref{eq:log-coefficient-compositions} to the free series
therefore gives
\begin{equation}\label{eq:free-log-composition}
 \frac{d_n^m}{m}
 =
 \sum_{k=1}^m\frac{(-1)^{k+1}}{k}
 \sum_{\substack{j_1+\cdots+j_k=m\\ j_\nu\geq1}}
 \prod_{\nu=1}^k d_n^{j_\nu}.
\end{equation}

Subtracting \eqref{eq:free-log-composition} from
\eqref{eq:log-coefficient-compositions}, we obtain
\begin{equation}\label{eq:log-difference}
 [t^m]\log H_n(t)-\frac{d_n^m}{m}
 =
 \sum_{k=1}^m\frac{(-1)^{k+1}}{k}
 \sum_{\substack{j_1+\cdots+j_k=m\\ j_\nu\geq1}}
 \left(
 \prod_{\nu=1}^k c_{j_\nu,n}
 -
 d_n^m
 \right).
\end{equation}
Fix a composition $(j_1,\ldots,j_k)$ of $m$.  Since
$c_{j_\nu,n}=d_n^{j_\nu}-\varepsilon_{j_\nu,n}$, expansion of the
product gives
\[
 \prod_{\nu=1}^k c_{j_\nu,n}-d_n^m
 =
 \sum_{\varnothing\neq S\subseteq\{1,\ldots,k\}}
 (-1)^{|S|}
 \left(\prod_{\nu\in S}\varepsilon_{j_\nu,n}\right)
 d_n^{\,m-\sum_{\nu\in S}j_\nu}.
\]
Terms for which some $j_\nu=1$ and $\nu\in S$ vanish because
$\varepsilon_{1,n}=0$.  Thus we may assume $j_\nu\geq2$ for every
$\nu\in S$.  If $|S|=s$, then \eqref{eq:epsilon-bound} gives
\begin{align*}
 \left|
 \left(\prod_{\nu\in S}\varepsilon_{j_\nu,n}\right)
 d_n^{\,m-\sum_{\nu\in S}j_\nu}
 \right|
 &\ll_m
 \left(\prod_{\nu\in S}
 r_nd_n^{j_\nu-2}\right)
 d_n^{\,m-\sum_{\nu\in S}j_\nu} \\
 &=
 r_n^s d_n^{m-2s}.
\end{align*}
By Proposition~\ref{prop:dr-growth}, $r_n=O(d_n)$ and
$d_n\to\infty$.  Hence, for every fixed $s\geq1$,
\begin{equation}\label{eq:multiple-error-bound}
 r_n^s d_n^{m-2s}
 =
 r_nd_n^{m-2}
 \left(\frac{r_n}{d_n^2}\right)^{s-1}
 =
 O_m(r_nd_n^{m-2}),
\end{equation}
where we have also used
$r_n/d_n^2=O(p^{-n})$ from \eqref{eq:r-over-d2}.

For fixed $m$, there are only finitely many compositions of $m$ and
only finitely many subsets $S$ occurring in the above expansions;
their total number depends only on $m$.  Combining
\eqref{eq:log-difference} with
\eqref{eq:multiple-error-bound} therefore gives
\[
 [t^m]\log H_n(t)-\frac{d_n^m}{m}
 =
 O_m(r_nd_n^{m-2}),
\]
which proves \eqref{eq:log-H}.

Finally, since $H_n(t)=1+d_nt+O(t^2)$, one has
$[t]\log H_n(t)=d_n$, proving the degree-one assertion.
\end{proof}

\begin{theorem}
\label{thm:zassenhaus-free}
Assume that $\mu>0$.  Then, for every fixed $m\geq2$, as
$n\to\infty$,
\begin{equation}\label{eq:zass-free}
 a_{m,n}
 =
 w_m^{(p)}(d_n)
 +O_m(r_nd_n^{m-2}).
\end{equation}
In particular,
\begin{equation}\label{eq:zass-leading}
 a_{m,n}
 =
 \frac{\delta^m}{m}p^{mn}
 +O_m(p^{(m-1)n}).
\end{equation}
For $m=1$, one has $a_{1,n}=d_n$.
\end{theorem}

\begin{proof}
For each $m\geq1$, put
\[
 b_{m,n}=[t^m]\log H_n(t).
\]
We first record explicitly the recursion obtained from the
Jennings--Lazard identity.  By \eqref{eq:jennings},
\[
 H_n(t)
 =
 \prod_{j\geq1}
 \left(\frac{1-t^{pj}}{1-t^j}\right)^{a_{j,n}}.
\]
Since $H_n(0)=1$, its formal logarithm is well defined in
$\Q[[t]]$.  Using
\[-\log(1-u)=\sum_{k\geq1}u^k/k,\] we obtain
\[
 \log H_n(t)
 =
 \sum_{j\geq1}a_{j,n}
 \left(
 \sum_{k\geq1}\frac{t^{jk}}{k}
 -
 \sum_{k\geq1}\frac{t^{pjk}}{k}
 \right).
\]
For a fixed $m$, only finitely many pairs $(j,k)$ contribute to the
coefficient of $t^m$.  In the first double sum, a contribution occurs
precisely when $jk=m$, and its contribution is
$a_{j,n}/k=ja_{j,n}/m$.  In the second sum, a contribution occurs
precisely when $pjk=m$.  Thus
\begin{equation}\label{eq:jennings-log-recursion}
 m b_{m,n}
 =
 \sum_{j\mid m}j\,a_{j,n}
 -
 p\mathbf 1_{p\mid m}
 \sum_{j\mid m/p}j\,a_{j,n}.
\end{equation}

We compare this recursion with the corresponding identity for the
free pro-$p$ group on $d_n$ generators.  Its Hilbert series is
$(1-d_nt)^{-1}$, and therefore
\[
 [t^m]\log\frac{1}{1-d_nt}
 =
 \frac{d_n^m}{m}.
\]
Its $j$-th Zassenhaus dimension is
$w_j^{(p)}(d_n)$.  Hence
\begin{equation}\label{eq:free-jennings-log-recursion}
 d_n^m
 =
 \sum_{j\mid m}j\,w_j^{(p)}(d_n)
 -
 p\mathbf 1_{p\mid m}
 \sum_{j\mid m/p}j\,w_j^{(p)}(d_n).
\end{equation}

Set
\[
 \Delta_{j,n}
 =
 a_{j,n}-w_j^{(p)}(d_n).
\]
Subtracting \eqref{eq:free-jennings-log-recursion} from
\eqref{eq:jennings-log-recursion} gives
\begin{equation}\label{eq:zass-difference-recursion}
 m\Delta_{m,n}
 =
 \bigl(mb_{m,n}-d_n^m\bigr)
 -
 \sum_{\substack{j\mid m\\j<m}}j\,\Delta_{j,n}
 +
 p\mathbf 1_{p\mid m}
 \sum_{j\mid m/p}j\,\Delta_{j,n}.
\end{equation}
Notice that every index occurring in the last two sums is strictly
smaller than $m$.

We now proceed by induction on $m$.  In degree one,
\[
 a_{1,n}
 =
 \dim_{\Fp}D_1(G_n)/D_2(G_n)
 =
 d_n,
\]
because $D_2(G_n)=\Phi(G_n)$, while
$w_1^{(p)}(d_n)=d_n$.  Hence $\Delta_{1,n}=0$.

Fix $m\geq2$ and suppose that
\[
 \Delta_{j,n}
 =
 O_j(r_nd_n^{j-2})
\]
has been proved for every $2\leq j<m$.  By
Lemma~\ref{lem:log-H},
\[
 mb_{m,n}-d_n^m
 =
 O_m(r_nd_n^{m-2}).
\]
For every divisor $j<m$ with $j\geq2$, the induction hypothesis gives
\[
 j\,\Delta_{j,n}
 =
 O_m(r_nd_n^{j-2})
 =
 O_m(r_nd_n^{m-2}),
\]
where in the last step we use $d_n\geq1$, which holds for all
sufficiently large $n$.  The same estimate applies to every term in
the final sum of \eqref{eq:zass-difference-recursion}; its indices
satisfy $j\leq m/p<m$.  The terms with $j=1$ vanish because
$\Delta_{1,n}=0$.  Since, for fixed $m$, only finitely many divisors
occur, \eqref{eq:zass-difference-recursion} gives
\[
 m\Delta_{m,n}
 =
 O_m(r_nd_n^{m-2}).
\]
Dividing by the fixed integer $m$ proves
\eqref{eq:zass-free}.

It remains to deduce the stated leading term.  By
Lemma~\ref{lem:free-restricted-lie-dimensions}, for fixed $m$,
\[
 w_m^{(p)}(d_n)
 =
 \frac{d_n^m}{m}+O_m(d_n^{m-1}).
\]
Moreover, Proposition~\ref{prop:dr-growth} gives $r_n=O(d_n)$.
Consequently, \eqref{eq:zass-free} yields
\[
 a_{m,n}
 =
 \frac{d_n^m}{m}+O_m(d_n^{m-1}).
\]
Since $\mu>0$, one has $\delta>0$ and
$d_n=\delta p^n+O(1)$.  For fixed $m$,
\[
 d_n^m
 =
 \delta^m p^{mn}+O_m(p^{(m-1)n})
 \quad\text{and}\quad
 d_n^{m-1}=O_m(p^{(m-1)n}).
\]
Substituting these estimates proves \eqref{eq:zass-leading}.
\end{proof}

The second Zassenhaus degree admits a more precise description which
does not require the assumption $\mu>0$.

\begin{proposition}
\label{prop:zass2}
For each $n$, put
\begin{equation}\label{eq:q2-def}
 q_{2,n}=d_n^2-c_{2,n}.
\end{equation}
Then
\begin{equation}\label{eq:q2-bound}
 0\leq q_{2,n}\leq r_n.
\end{equation}
Since $p$ is odd, one furthermore has the exact identity
\begin{equation}\label{eq:a2-exact}
 a_{2,n}
 =
 \binom{d_n}{2}-q_{2,n}.
\end{equation}
Consequently,
\begin{equation}\label{eq:a2-sandwich}
 \binom{d_n}{2}-r_n
 \leq
 a_{2,n}
 \leq
 \binom{d_n}{2}.
\end{equation}
\end{proposition}

\begin{proof}
Taking $m=2$ in Lemma~\ref{lem:fixed-degree}, we have
$S_2(d_n)=1$.  Therefore
\[
 d_n^2-r_n
 \leq
 c_{2,n}
 \leq
 d_n^2.
\]
Since $q_{2,n}=d_n^2-c_{2,n}$, this is equivalent to
\eqref{eq:q2-bound}.

We now recover the second Zassenhaus dimension from $c_{2,n}$.
Since
\[
 H_n(t)
 =
 1+d_nt+c_{2,n}t^2+O(t^3),
\]
the formal expansion of the logarithm gives
\begin{equation}\label{eq:log-H-degree-two}
 b_{2,n}
 =
 [t^2]\log H_n(t)
 =
 c_{2,n}-\frac{d_n^2}{2}.
\end{equation}
On the other hand, because $p$ is odd, one has $p\nmid2$.
Taking $m=2$ in the Jennings logarithmic recursion
\eqref{eq:jennings-log-recursion} therefore gives
\[
 2b_{2,n}
 =
 a_{1,n}+2a_{2,n}.
\]
Since $a_{1,n}=d_n$, it follows that
\[
 a_{2,n}
 =
 b_{2,n}-\frac{d_n}{2}.
\]
Using \eqref{eq:log-H-degree-two}, we obtain
\[
 a_{2,n}
 =
 c_{2,n}-\frac{d_n^2+d_n}{2}.
\]
Finally, substituting
$c_{2,n}=d_n^2-q_{2,n}$ gives
\[
 a_{2,n}
 =
 \frac{d_n^2-d_n}{2}-q_{2,n}
 =
 \binom{d_n}{2}-q_{2,n},
\]
which proves \eqref{eq:a2-exact}.  Combining this identity with
$0\leq q_{2,n}\leq r_n$ gives
\eqref{eq:a2-sandwich}.
\end{proof}

\printbibliography

\end{document}